\documentclass[11pt,reqno]{amsart}

\usepackage{graphicx}
\usepackage{verbatim}
\usepackage{textcomp}
\usepackage{amssymb}
\usepackage{cite}
\usepackage{amsmath}
\usepackage{latexsym}
\usepackage{amscd}
\usepackage{amsthm}
\usepackage{mathrsfs}
\usepackage{xypic}
\usepackage{bm}
\usepackage{url}
\usepackage{hyperref}
\usepackage{color}

\usepackage{geometry}
\newtheorem{thm}{Theorem}[section]

\newtheorem{lem}[thm]{Lemma}
\newtheorem{prop}[thm]{Proposition}

\theoremstyle{definition}
\newtheorem{defn}{Definition}[section]

\newtheorem{rem}[thm]{Remark}
\numberwithin{equation}{section}
\def\C{\mathbb C}
\def\R{\mathbb R}

\def\cal{\mathcal}

\def\pt{\partial}
\def\p{\partial}
\def\S{\Sigma}
\def\<{\left<}
\def\>{\right>}
\def\g{\gamma}
\def\O{\Omega}

\def\ve{\varepsilon}

\def\De{\Delta}
\def\n{\nabla}

\def\l{\lambda}
\def\S{\Sigma}
\def\<{\langle}
\def\>{\rangle}

\def\s{\sigma}

\begin{document}

\title[Escobar's Conjecture]{Escobar's Conjecture on a sharp lower bound for the first nonzero Steklov eigenvalue}
\author{Chao Xia}
\address{School of Mathematical Sciences\\
Xiamen University\\
361005, Xiamen, P.~R. China}
\email{chaoxia@xmu.edu.cn}
\author{Changwei~Xiong}
\address{School of Mathematics, Sichuan University, Chengdu 610065, Sichuan,  P.~R.~China}
\email{changwei.xiong@scu.edu.cn}
\thanks{C.~Xia is supported by  NSFC (Grant no. 11871406, 12271449). C.~Xiong is supported by Australian Laureate Fellowship FL150100126 of the Australian Research Council, National Key R and D Program of China 2021YFA1001800 and NSFC (Grant no. 12171334).}


\maketitle

\begin{abstract}

 It was conjectured by Escobar [J. Funct. Anal. \textbf{165} (1999), 101--116] that for an $n$-dimensional ($n\geq 3$) smooth compact Riemannian manifold with boundary, which has nonnegative Ricci curvature and boundary principal curvatures bounded below by $c>0$, the first nonzero Steklov eigenvalue is greater than or equal to $c$ with equality holding only on isometrically Euclidean balls with radius $1/c$. In this paper, we confirm this conjecture in the case of nonnegative sectional curvature.  The proof is based on a combination of Qiu--Xia's weighted Reilly-type formula with a special choice of the weight function depending on the distance function to the boundary, as well as a generalized Pohozaev-type identity. 
\end{abstract}

\section{Introduction}

Let $(\Omega^n, g)$ be an $n$-dimensional ($n\geq 2$) smooth compact connected Riemannian manifold with boundary $\S=\pt \O$. We are interested in the Steklov eigenvalue problem on $\O$, introduced by Steklov in 1895 (see \cite{KKK14}, \cite{Ste02}):
\begin{equation}\label{Steklov}
\begin{cases}
\Delta f=0,&\text{ in } \O,\\
\dfrac{\partial f}{\partial \nu}=\sigma f, &\text{ on } \S,
\end{cases}
\end{equation}
where $\Delta$ denotes the Laplace--Beltrami operator of $\O$ and $\nu$ is the outward unit normal along $\S$. Equivalently, the Steklov eigenvalues constitute the spectrum of the Dirichlet-to-Neumann map $\Lambda:C^\infty(\S)\rightarrow C^\infty(\S)$ defined by
\begin{equation*}
\Lambda f=\frac{\pt (\mathcal{H}f)}{\pt \nu},\quad f\in C^\infty(\S),
\end{equation*}
where $\mathcal{H}f$ is the harmonic extension of $f$ to the interior of $\O$. The Dirichlet-to-Neumann map $\Lambda$ is a first-order elliptic pseudo-differential operator \cite[pp. 37--38]{Tay96} and its spectrum is nonnegative, discrete and unbounded (counted with multiplicities):
\begin{equation*}
0=\sigma_0<\sigma_1\leq \sigma_2\leq \cdots \nearrow +\infty.
\end{equation*}
A standard variational principle for the first nonzero Steklov eigenvalue is given by
\begin{equation}\label{var-ch}
\sigma_1=\inf_{f\in C^1(\S),\\ \int_\S f da=0}\frac{\int_\Omega |\nabla (\mathcal{H}f)|^2 dv}{\int_\S f^2 da}.
\end{equation}
We refer to the excellent survey \cite{GP17} for an account of the Steklov eigenvalue problem.

In this paper we are mainly concerned with the sharp lower bound for the first nonzero Steklov eigenvalue $\sigma_1$. 

\subsection{Sharp lower bound of the first nonzero Steklov eigenvalue}\

In 1970, Payne \cite{Pay70} proved that for a convex planar domain $\O\subset\mathbb{R}^2$ whose boundary curve has its geodesic curvature $\ge c>0$, its first nonzero Steklov eigenvalue satisfies $\s_1\ge c$ with equality holding only for a round disk with radius $1/c$. This sharp lower bound for $\s_1$ has been generalized by Escobar \cite{Esc97} to non-negatively curved $2$-dimensional manifolds. Both of Payne's and Escobar's approaches, which are based on the maximum principle, work only for the $2$-dimensional case. In higher dimensions, a non-sharp lower bound $\s_1>c/2$ has been established by Escobar \cite{Esc97} for $n$-dimensional manifolds with nonnegative Ricci curvature and boundary principal curvatures $\ge c$ by using Reilly's formula \cite{Rei77}. Based on the above results, Escobar raised the following conjecture in 1999.

\medskip

\noindent{\bf Escobar's Conjecture \cite{Esc99}.} Let $(\Omega^n,g)$ be an $n$-dimensional ($n\geq 3$) smooth compact connected Riemannian manifold with boundary $\S=\pt \O$. Assume that $${\rm Ric}_g\ge 0, \hbox{ and }h\ge cg_\S>0\hbox{ on }\S.$$ Then $\s_1\ge c$ with equality holding only for a Euclidean ball of radius $ 1/ c$.

\medskip

Here and throughout the paper, we denote by ${\rm Ric}_g$ the Ricci curvature $2$-tensor for $(\O, g)$ and by $h$ the second fundamental form of $\S$. For notational simplicity, we use ${\rm Ric}_g\ge 0$, ${\rm Sect}_g\ge 0$ and $h\ge cg_\S$ to indicate that $(\O, g)$ has nonnegative Ricci curvature, nonnegative sectional curvature and $\S$ has its principal curvatures $\ge c$ respectively.

There has been little progress since Escobar raised this conjecture. Monta\~{n}o \cite{Mon13} showed in 2013 that the conjecture is true for rotationally symmetric metrics  (see \cite{Xio19} for a different proof).  In fact, there is not much difference in techniques between $2$-dimensional general metrics and higher dimensional rotationally symmetric metrics. He also checked in \cite{Mon16} that Escobar's conjecture is true for Euclidean ellipsoids.

In this paper, we confirm Escobar's conjecture for manifolds with nonnegative sectional curvature.
\begin{thm}\label{thm1}
Let $(\Omega^n,g)$ be an $n$-dimensional ($n\geq 2$) smooth compact connected Riemannian manifold with boundary $\S=\pt \O$. Assume that \begin{eqnarray}\label{curv-cond}
{\rm Sect}_g\ge 0, \hbox{  and }h\ge cg_\S>0 \hbox{ on }\S.
\end{eqnarray}
Then the first nonzero Steklov eigenvalue $\sigma_1$ for $\O$ satisfies
\begin{equation*}
\sigma_1\geq c,
\end{equation*}
with equality if and only if $\O$ is isometric to a Euclidean ball with radius $1/c$.
\end{thm}

A special case is when $\O$ is a bounded domain in $\mathbb{R}^n$. We list it below separately because of its significance.
\begin{thm}\label{cor1}
Let $\Omega\subset\mathbb{R}^n$ ($n\geq 2$) be a smooth bounded domain in $\mathbb{R}^n$. Assume that the principal curvatures of $\S=\p\O$ are bounded below by $c>0$. Then the first nonzero Steklov eigenvalue $\sigma_1$ for $\O$ satisfies
\begin{equation*}
\sigma_1\geq c,
\end{equation*}
with equality if and only if $\O$ is a Euclidean ball with radius $1/c$.
\end{thm}

In addition, in view of the variational characterization \eqref{var-ch} for $\sigma_1$, our result is equivalent to a sharp Poincar\'{e}-trace inequality. It is also worth mentioning that our result can be viewed as a sharp lower bound of the fundamental gap for the Steklov eigenvalue problem (noting $\sigma_1-\sigma_0=\sigma_1$); for the sharp lower bounds on the fundamental gaps of the Dirichlet and the Neumann eigenvalue problems, we refer to \cite{AC11,Kro92,PW60,ZY84}.

We use a method which is totally based on integral identities and inequalities to prove Theorem \ref{thm1}. In particular, the proof has two main ingredients. One is a weighted Reilly-type formula proved by Qiu and the first-named author \cite{QX15}, with a special choice of the weight function
\begin{equation}\label{special-weight}
V=\rho-\frac c2\rho^2,
\end{equation}
where $\rho=\mathrm{dist}(\cdot,\S)$ is the distance function  to $\S$. The other is a generalized Pohozaev-type identity which was proved by Provenzano--Stubbe \cite{PS19} for Euclidean domains and by the second-named author \cite{Xio18} for general manifolds, with a special choice of the gradient vector field $\n V$ in the identity. Such a Pohozaev-type identity has been recently used to obtain bounds on Steklov eigenvalues; see e.g. \cite{CGH20} and \cite{GKLP22}. Remarkably, in spite of the use of the weighted Reilly-type formula, after using the Hessian comparison theorem, we arrive at two key inequalities \eqref{key-ineq1} and \eqref{key-ineq2}, establishing the relations among the interior Dirichlet integral, the boundary Dirichlet integral and the boundary $L^2$ norm of the normal derivative for a harmonic function $f$, which does not involve the weight function $V$.
Our argument works for all dimensions. Hence it also provides a new proof for the $2$-dimensional case.

For our purpose, the crucial property of $V$ is the following Hessian comparison result. The curvature condition \eqref{curv-cond} implies that
\begin{eqnarray}\label{hess-comp}
\nabla^2 V\le -cg
\end{eqnarray} holds true away from ${\rm Cut}(\S)$, the cut locus of $\S$ in $\O$. This follows directly from the Hessian comparison theorem for $\rho=\mathrm{dist}(\cdot,\S)$ implicitly given by Heintze-Karcher \cite[Section 3.2]{HK78}; see also Kasue \cite[Remark 2.26]{Kas82}. Moreover, Kasue \cite{Kas82} proved that \eqref{hess-comp} holds true throughout $\O$ in the weak sense of Wu \cite{Wu79}.
Since the weight function $V$  is only Lipschitz continuous on  ${\rm Cut}(\S)$, in order to apply Qiu--Xia's weighted Reilly-type formula and the generalized Pohozaev-type identity,  we have to make a smooth  approximation $V_\varepsilon\in C^\infty(\O)$ of $V$. Fortunately, we are able to choose a Greene--Wu type smooth approximation $V_\varepsilon$ of $V$ which is identical to $V$ near $\S$ and satisfies $\nabla^2 V_\varepsilon\le -(c-\varepsilon) g$ for any small $\varepsilon>0$. This is the main technical part in the proof.



\subsection{Relation for the spectra of two eigenvalue problems}\

As a byproduct of our argument, we are able to provide some new results on the comparison between the spectrum of the Steklov eigenvalue problem on $(\O,g)$ and that of the Laplacian eigenvalue problem on its boundary $\S$.

Let $\Delta_\S$ denote the Laplace--Beltrami operator acting on smooth functions on $\S$.
 The  spectrum of $\S$ (for $\Delta_\S$) 
consists of an increasing discrete sequence of nonnegative eigenvalues (counted with multiplicities)
\begin{equation*}
0=\lambda_0<\lambda_1\leq \lambda_2\leq \cdots \nearrow +\infty.
\end{equation*}
There are various types of comparison between the Steklov eigenvalue $\sigma_j$ and the Laplacian eigenvalue $\lambda_j$. See e.g. \cite{WX09, Kar17,PS19,Xio18,Esc99}  and the references therein. Among them the most relevant works to our result here are Q.~Wang and C.~Xia's \cite{WX09} and M.~Karpukhin's \cite{Kar17}.

Q.~Wang and C.~Xia \cite{WX09} proved that for Riemannian manifolds of dimension $n\geq 2$ with $\mathrm{Ric}_g\geq 0$ and boundary principal curvatures $\ge c$, there holds
\begin{equation}\label{wx}
\sigma_1\leq \frac{\sqrt{\lambda_1}}{(n-1)c}(\sqrt{\lambda_1}+\sqrt{\lambda_1-(n-1)c^2})
\end{equation}
with equality holding only for Euclidean balls with radius $1/c$. Recently, based on the previous results of Raulot--Savo \cite{RS12} and Yang--Yu \cite{YY17} on estimates of the Steklov eigenvalue for differential forms, M.~Karpukhin \cite{Kar17} showed  that for Riemannian manifolds of dimension $n\geq 3$ with nonnegative second Weitzenb\"{o}ck curvature $W^{[2]}$  and boundary $(n-2)$-curvature $\ge (n-2)c$, there holds for $j\geq 1$,
\begin{align}
\sigma_j&\leq \frac{\lambda_j}{(n-1)c}, \text{ when } n\geq 4;\label{Kar1}\\
& \sigma_j< \frac{2\lambda_j}{3c},\text{ when } n=3.\label{Kar2}
\end{align}
See \cite{Kar17} for the precise definitions of the Weitzenb\"{o}ck curvature and the boundary $(n-2)$-curvature. 

In this paper we add new results of the same type for Riemannian manifolds of nonnegative sectional curvature and strictly convex boundary. Precisely, we prove Theorem~\ref{thm2} below.
\begin{thm}\label{thm2}
Let $(\Omega^n,g)$ be as in Theorem \ref{thm1}. 
Then the first nonzero Steklov eigenvalue $\sigma_1$ for $\O$ and the first nonzero eigenvalue $\lambda_1$ for $\S$ 
satisfy
\begin{equation}\label{upper}
\sigma_1\leq \frac{\lambda_1}{(n-1)c},
\end{equation}
with equality if and only if $\O$ is isometric to a Euclidean ball with radius $1/c$. Moreover, the $j$th Steklov eigenvalue $\sigma_j$ for $\O$ and the $j$th eigenvalue $\lambda_j$ for $\S$ satisfy
\begin{equation}\label{high}
\sigma_j\leq \frac{\lambda_j}{(n-1)c},\quad j\geq 2.
\end{equation}
For $2\le j\le n$, the equality in \eqref{high} is achieved by Euclidean balls with radius $1/c$.
\end{thm}
\begin{rem}\label{rem-connect}
By the results in \cite{Ich81,Kas83}, any compact Riemannian manifold with nonnegative Ricci curvature and strictly mean convex boundary must have only one boundary component. So the boundary $\S$ of the Riemannian manifold $\O$ in Theorem~\ref{thm2} is connected, which shows that $0$ is an eigenvalue for $\Delta_\S$ of multiplicity one.
\end{rem}

\begin{rem}
Let us compare Theorem \ref{thm2} with Wang--Xia's \eqref{wx} and Karpukhin's \eqref{Kar1} and \eqref{Kar2}. First, compared with Wang--Xia's \eqref{wx}, our estimate~\eqref{upper} is better; however, note that our assumption $\mathrm{Sect}_g\geq 0$ is stronger than theirs, $\mathrm{Ric}_g\geq 0$. Second, for $n=3$, due to the duality induced by the Hodge $*$-operator, $W^{[2]}\geq 0$ is equivalent to $W^{[1]}=\mathrm{Ric}_g\geq 0$. So in this case our assumption $\mathrm{Sect}_g\geq 0$ is stronger than Karpukhin's $W^{[2]}\geq 0$ (the boundary assumptions are the same), while our estimates \eqref{upper} and \eqref{high} are better than his \eqref{Kar2}. For $n\geq 4$, our \eqref{upper} and \eqref{high} are the same as Karpukhin's \eqref{Kar1}. Nevertheless, in this case, to the best of our knowledge there is no direct relation between $\mathrm{Sect}_g\geq 0$ and $W^{[2]}\geq 0$. For instance, when $n=4$, the condition $W^{[2]}\geq 0$ is equivalent to the isotropic curvature being nonnegative. 
(The concept of isotropic curvature was introduced by Micallef--Moore \cite{MM88} and the relation between nonnegative $W^{[2]}$ and nonnegative isotropic curvature was investigated by Micallef--Wang \cite{MW93} and others; see e.g. Thm.~2.1~(a) in \cite{MW93}, Chap.~9 in \cite{Pet16}, or Prop.~3.3 in \cite{Nor94}.) On the other hand, the conditions of nonnegative sectional curvature and nonnegative isotropic curvature are not mutually inclusive. It is well-known that  the Fubini--Study metric on the complex projective space $\C \mathbb{P}^2$ has sectional curvature lying in the interval $[1,4]$ and has nonnegative isotropic curvature but not positive isotropic curvature; see e.g. \cite{MW93}. Therefore, a small perturbation of  the Fubini--Study metric on $\C \mathbb{P}^2$ yields an example which satisfies $\mathrm{Sect}_g\geq 0$ but admits negative isotropic curvature somewhere.
\end{rem}




The proof of Theorem~\ref{thm2} is based on Qiu--Xia's weighted Reilly-type formula \cite{QX15} with the same choice of the weight function $V$ as in Theorem~\ref{thm1}.

The rest of the paper is structured as follows. In Section~\ref{sec2} we recall two integral formulas. One is the Qiu--Xia's weighted Reilly-type formula, and the other is a generalized Pohozaev-type identity.  In Section \ref{sec3}, we first recall the Hessian comparison of the distance function to the boundary and then carry out a smoothing procedure on the weight function $V$ defined in \eqref{special-weight}. In Section~\ref{sec4} we present the proofs of Theorems~\ref{thm1} and \ref{thm2}.

\

\section{Weighted Reilly formula and Pohozaev identity}\label{sec2}

At the beginning of this section, we fix our notations.
Let $(\Omega^{n}, g)$ be an $n$-dimensional compact  Riemannian manifold with smooth boundary $\S$. Let $g_\S$ be the induced metric of $\S$. We use $\<\cdot, \cdot\>$ to denote the inner product with respect to both $g$ and $g_\S$ when no confusion occurs. We denote by $\nabla$, $\Delta$ and $\n^2$ the gradient, the Laplacian and the Hessian on $\O$ respectively, while by $\n_\S$ and $\De_\S$ the gradient and the Laplacian on $\S$ respectively.  Let $\nu$ be the unit outward normal of $\S$. We denote by $h(X,Y)=g(\n_X \nu, Y)$ and $H={\rm tr}_{g_\S} h$ the second fundamental form and the mean curvature of $\S$ respectively. Let $dv$ and $da$ be the canonical volume element of $\O$ and $\S$ respectively. Let ${\rm Ric}_g$ be the Ricci curvature tensor of $\O$.


We recall the following weighted Reilly-type formula proved by Qiu and the first-named author (See \cite[Thm.~1.1]{QX15} in the case $K=0$).
\begin{prop}[\cite{QX15}]
For two smooth functions $f$ and $V$ on ${\O}$, we have
\begin{align}\label{xeq0}
&\int_\O V\left((\Delta f)^2-|\n^2 f|^2\right)dv\nonumber\\
=&\int_\S V\left[2\pt_\nu f\De_\S f+H(\pt_\nu f)^2+h(\n_\S f, \n_\S f)\right]da\nonumber \\
&{}+\int_\S \pt_\nu V\, |\n_\S f|^2  da +\int_\O \left(\n^2 V -\Delta V   g+ V \mathrm{Ric}_g\right)(\n f, \n f)dv.
\end{align}
\end{prop}

We also need the following generalized Pohozaev  type identity (see \cite[Lem.~9]{Xio18}).
\begin{prop}[\cite{PS19,Xio18}]
For a smooth vector field $X$ and a harmonic function $f$ on $\O$, we have
\begin{align}\label{poho}
\int_\O \left(\<\n_{\n f}X, \n f\> -\frac12|\n f|^2  {\rm div}_gX \right)dv = \int_\S \left(\p_\nu f \<X, \n f\>-\frac12|\n f|^2 \<X,\nu\>\right) da.\end{align}
\end{prop}

\

\section{Distance function to the boundary}\label{sec3}

In this section, we study the distance function to the boundary and its Greene--Wu type smooth approximation.

\subsection{Hessian comparison of the distance function to the boundary}\

Following the terminology of \cite{Wu79} and \cite{Kas82}, for any continuous function $f\in C(\O)$, we introduce an extended real number $Cf(x;X)$ for a point $x\in \O$ and $X\in T_x\O$ by
\begin{equation}\label{def-cf}
Cf(x;X)=\liminf_{r\rightarrow 0}\frac{f(\exp_x(rX))+f(\exp_x(-rX))-2f(x)}{r^2}.
\end{equation}
When $f\in C^2$, we have $Cf(x;X)=\n^2f|_x(X,X).$

Define the distance function to the boundary $\S$ by
\begin{equation*}
\rho=\rho(x)={\rm dist}(x,  \S).
\end{equation*}
The distance function $\rho$ is smooth away from the cut locus ${\rm Cut}(\S)$ of $\S$. Recall that ${\rm Cut}(\S)$ is defined to be the set of all cut points and a cut point is the first point on a normal geodesic initiating from the boundary $\S$ at which this geodesic fails to minimize  uniquely for the distance function $\rho$. In other words, for $x\in \S$, consider the arc-length parametrized geodesic $\g_x(t)=\exp_x(-t\nu(x))$ ($t\geq 0$). Then
$\g_x(t_0)\in {\rm Cut}(\S)$ for
\begin{equation*}
t_0=t_0(x)=\sup\{t>0:\mathrm{dist}(\g_x(t),\S)=t\}.
\end{equation*}

The set ${\rm Cut}(\S)$ is known to have zero $n$-dimensional Hausdorff measure; see e.g. \cite[Thm.~B]{IT01}. In addition, under the curvature conditions \eqref{curv-cond}, we have
\begin{eqnarray}\label{dist-max}
\rho_{\max}=\max_{ \O} \rho\leq \frac{1}{c}.
\end{eqnarray}
See e.g. \cite{Li14}.

We recall the following Hessian comparison theorem for the distance function $\rho$ (\cite[Thm.~2.31]{Kas82}).
\begin{thm}[\cite{Kas82}]\label{kasue}
Let $\psi:(0,\rho_{\max}]\rightarrow \R$ be a $C^2$ nonincreasing function. For $x\in \O$, let $\gamma:[0,l]\rightarrow \O$ be a unit speed geodesic joining $\S$ and $x$ such that $\mathrm{dist}(\gamma(t),\S)=t$ for $t\in [0,l]$ and $\gamma(l)=x$. Then for $X\in T_x \O$, we have
\begin{align*}
C(\psi(\rho))(x;X)\geq \left(\psi''\langle \gamma'(l),X\rangle^2+\psi' \frac{\Theta'}{\Theta}(|X|^2-\langle \gamma'(l),X\rangle^2)\right)(\rho(x)),
\end{align*}
where $\Theta(t)\in C^2([0,l])$ satisfies $\Theta''(t)+K(t)\Theta(t)=0$ with $\Theta(0)=1$ and $\Theta'(0)\geq -c$. Here $K(t)$ is a lower bound of the sectional curvature at $\gamma(t)$ of the planes containing $\gamma'(t)$ and $c$ is a lower bound of the principal curvatures of $\S$ at $\gamma(0)$.
\end{thm}

We shall apply the above theorem to $ -V(\rho)$, where $V(\rho)$ is given by
\begin{align}\label{v-def}
V=V(\rho)=\rho-\frac{c}{2}\rho^2.
\end{align}
It is easy to see that $V>0$ and $-V$ is a nonincreasing function thanks to \eqref{dist-max}. We remark that $V(\rho)$ should be compared with the function $\eta$ defined in \cite{PS19} and \cite{Xio18} which also depends on the distance function $\rho$ to the boundary. First, $\eta$ is defined only on a tubular neighborhood of the boundary, while $V$ is on the whole $\O$. Second, the $\eta$ which can be compared, is defined for Riemannian manifolds with nonpositive sectional curvature, while $V$ is for those with nonnegative sectional curvature.

By choosing $K(t)=0$ and $\Theta(t)=1-ct$ in Theorem \ref{kasue}, and noting that $V'(\rho)=\Theta(\rho)$, we find the following comparison for $-V$.
\begin{prop}
Let $(\O, g)$ be as in Theorem \ref{thm1} and $V$ be defined by \eqref{v-def}. Then
\begin{equation}
C(-V(\rho))(x;X)\geq c
\end{equation}
for any $x\in \O$ and any unit vector $X\in T_x\O$.
\end{prop}

\subsection{Smoothing of the distance function}\

We shall use \eqref{xeq0} and \eqref{poho} with $V$ involved. However, the function $V(\rho)$ is not smooth on ${\rm Cut}(\S)$ so that we cannot apply \eqref{xeq0} and \eqref{poho} directly to $V$. To overcome this problem, we construct a smooth Greene--Wu type approximation by the Riemannian convolution and a gluing procedure. More precisely, we have the following result.

\begin{prop}\label{prop-smoothing}
Fix a neighborhood $\cal{C}$ of ${\rm Cut}(\S)$ in $\O$. Then for any $\ve>0$, there exists a smooth nonnegative function ${V}_\ve$ on $\O$ such that $V_\ve=V$ on $\O\setminus \cal{C}$ and
\begin{equation}
\n^2 (-{V}_\ve)\geq (c-\ve)g.
\end{equation}
\end{prop}

The remaining of this section is devoted to the proof of Proposition~\ref{prop-smoothing}.

For notational convenience, let us write $O_2$ for $\cal{C}$ and choose two other neighborhoods $O_1$ and $O_3$ of ${\rm Cut}(\S)$ such that
\begin{equation}
O_{1}\subset\subset O_{2} \subset\subset O_{3}\subset\subset \O,
\end{equation}
where ``$A\subset\subset B$'' for two sets $A$ and $B$ means ``$\overline{A}\subset B$ and $\overline{A}$ is compact''.
We shall first mollify $V$ on $O_{3}$ by the standard Riemannian convolution.

Recall that the Riemannian convolution, introduced by Greene and Wu  \cite{GW72-73,GW76,GW79}, is defined by
\begin{align*}
\widetilde{V}_\tau(x)=\frac{1}{\tau^n}\int_{v\in T_xM}V(\exp_x(v))\theta\left(\frac{|v|}{\tau}\right)d\mu_x,
\end{align*}
where $\mu_x$ is the Lebesgue measure on $T_xM$ determined by the Riemannian metric $g$ at $x$ and $\theta$ is a smooth nonnegative function on $\R$ with support in $[-1,1]$ which is a positive constant in a neighborhood of $0$ and satisfies
\begin{equation*}
\int_{\R^n} \theta(|x|)dx=1.
\end{equation*}
In the following we assume $\tau<\mathrm{dist}(O_3,\S)$. So we get a smooth function $\widetilde{V}_\tau$ on $O_3$.

Next recall the following definition of convexity for continuous functions;  see for example  \cite[page 60]{GW79}.
\begin{defn}[\cite{GW79}]
Let $f:M\rightarrow \R$ be a continuous function on a Riemannian manifold $M$ and $\xi$ be a real number. The function $f$ is called $\xi$-convex at a point $p\in M$ if there is a positive number $\delta$ such that the function $q\mapsto f(q)-((\xi+\delta)/2)\mathrm{dist}^2(p,q)$ is convex in a neighborhood of $p$.
\end{defn}

Using this definition, we can prove the following result.
\begin{lem}\label{lem-convex}
The function $-V$ defined by \eqref{v-def} is $(c-\eta)$-convex  on $O_3$ for any $\eta>0$.
\end{lem}
\begin{proof}
Fix a point $p\in O_3$ and $\eta>0$. We need to prove there exist a neighborhood $U$ of $p$ and a positive number $\delta$ such that the function $$\varphi(q)=-V(q)-\frac{c-\eta+\delta}{2}\mathrm{dist}^2(p,q)$$ is convex on $U$. Equivalently, for any $q\in U$ and a unit $X\in T_qU$, we need to prove
\begin{equation*}
\varphi(\exp_q(rX))+\varphi(\exp_q(-rX))-2\varphi(q)\geq 0
\end{equation*}
for small $r>0$.

First since $C(-V)(y;Y)\geq c$ for any $y\in O_3$ and any unit $Y\in T_y\Omega$, and $\overline{O_3}$ is compact, by the definition of \eqref{def-cf}, we conclude that when $r$ is small enough,
\begin{align*}
-V(\exp_q(rX))-V(\exp_q(-rX))+2V(q)\geq \left(c-\frac{\eta}{2}\right)r^2.
\end{align*}

Then we deduce
\begin{align*}
&\quad \varphi(\exp_q(rX))+\varphi(\exp_q(-rX))-2\varphi(q)\\
&=-V(\exp_q(rX))-V(\exp_q(-rX))+2V(q) -\frac{c-\eta+\delta}{2}A(r) \\
&\geq  \left(c-\frac{\eta}{2}\right)r^2-\frac{c-\eta+\delta}{2}A(r),
\end{align*}
where $A(r):=\mathrm{dist}^2(p,\exp_q(rX))+\mathrm{dist}^2(p,\exp_q(-rX))-2\mathrm{dist}^2(p,q)$.

When $\O\subset \R^n$, we know exactly $A(r)=2r^2$. Now if we choose $U$ small enough, the metric on $U$ is close to the Euclidean metric. So for $U$ small, we have
\begin{align*}
0\leq A(r)\leq 2(1+\epsilon )r^2
\end{align*}
for small $\epsilon$ and small $r$.

As a consequence, we get
\begin{align*}
&\quad \varphi(\exp_q(rX))+\varphi(\exp_q(-rX))-2\varphi(q) \\
&\geq   \left(c-\frac{\eta}{2}\right)r^2-(c-\eta+\delta)(1+\epsilon)r^2\\
&=  \left(\frac{\eta}{2}-\delta-(c-\eta+\delta)\epsilon\right)r^2\geq  0,
\end{align*}
provided that $\delta$ and $\epsilon$ are chosen small enough. So we finish the proof.

\end{proof}

Next we need the approximation result \cite{GW72-73,GW76,GW79} for  $\xi$-convex functions by its Riemannian convolution. The case $\xi=0$ was considered in \cite{GW72-73,GW76}. The general case follows from the same proof; see pp.~60--61 in \cite{GW79}.
\begin{prop}[{\cite{GW72-73,GW76,GW79}}]\label{prop-appro}
If $f$ is a $\xi$-convex function on a Riemannian manifold $M$ and $K$ is a compact subset of $M$, then there exist a neighborhood of $K$ and a $\tau_0>0$ such that for all $\tau\in (0,\tau_0)$, the Riemannian convolution $\widetilde{f}_\tau$ of $f$ is $\xi$-convex on the neighborhood.
\end{prop}

For $\ve>0$, by Lemma~\ref{lem-convex} we know that $-V$ is $(c-\ve)$-convex on $O_3$. Applying Proposition~\ref{prop-appro} to $-V$ with $K=\overline{O_2}$, we have the following result.
\begin{lem}
For $\ve>0$, there exists a $\tau_0>0$ such that $-\widetilde{V}_\tau$ is $(c-\ve)$-convex on $O_{2}$ for $\tau\in (0,\tau_0)$.
\end{lem}
In particular, we get
\begin{equation*}
\n^2{(-\widetilde{V}_\tau)}|_x(X,X)=C(-\widetilde{V}_\tau)(x;X)\geq c-\ve
\end{equation*}
for $x\in O_{2}$ and any unit $X\in T_x\Omega$, provided that $\tau\in (0,\tau_0)$.

Next by a gluing procedure as in  \cite{Gho02} thanks to Ghomi, we can construct the desired function $V_\ve$ in Proposition \ref{prop-smoothing}. More precisely, let $\phi$ be a smooth nonnegative cut-off function such that $\mathrm{supp}\; \phi \subset O_{2}$ and $\phi\equiv 1$ on $O_1$ and define
\begin{equation}
V_\tau=\phi \widetilde{V}_\tau+(1-\phi)V,
\end{equation}
which gives us a smooth function on $\O$.

We claim that $V_\tau$ satisfies all the requirements in Proposition~\ref{prop-smoothing} when $\tau$ is small enough. In fact, on $\O\setminus O_{2}$ we have $V_\tau=V$. On $O_1$, we have $V_\tau=\widetilde{V}_\tau$, and
\begin{equation*}
\n^2(-V_\tau)\geq  (c-\ve)g.
\end{equation*}
Lastly consider $V_\tau$ on $\overline{O_{2}}\setminus O_1$.
Since $V$ is smooth on $\O\setminus {\rm Cut}(\S)$, by Lemma~3 (3) in \cite{GW76}, we see
\begin{equation*}
\lim_{\tau\rightarrow 0}\|\widetilde{V}_\tau-V\|_{C^2(\overline{O_2}\setminus O_1)}=0.
\end{equation*}

Note that $V_\tau-V=\phi\cdot (\widetilde{V}_\tau-V)$. So for $\ve>0$, there exists $\tau(\ve)>0$ such that
\begin{equation*}
\n^2(-V_{\tau(\ve)})(X,X) \ge \n^2(-V)(X,X)-\ve \geq c-\ve
\end{equation*}
for $x\in \overline{O_{2}}\setminus O_1$ and any unit vector $X\in T_x \O$.

We write simply $V_\ve=V_{\tau(\ve)}$. Finally, since $V>0$ on $\O$, by noting that the Riemannian convolution and the gluing procedure always keep the positivity, we see $V_\ve\ge 0$. The proof of Proposition~\ref{prop-smoothing} is completed.

\

\section{Proofs of Theorems~\ref{thm1} and \ref{thm2}}\label{sec4}

We shall prove the following two key  inequalities for harmonic functions on $\O$.
\begin{prop}
Let $(\O, g)$ be as in Theorem \ref{thm1}. Let $f$ be a harmonic function, i.e., $\De f=0$  on $\O$. Then we have
\begin{eqnarray}
&&\qquad \int_\S (\pt_\nu f)^2 da\ge c\int_\O |\n f|^2 dv,\label{key-ineq1}\\
&&\int_\S |\n_\S f|^2 da \ge (n-1)c\int_\O |\n f|^2 dv.\label{key-ineq2}
\end{eqnarray}
\end{prop}
\begin{proof}
By our construction of $V_\ve$, we have
\begin{align}\label{info-bdry}
&&V_\ve|_{\S}=V|_\S=0\hbox{ and }\n_\nu {V}_\ve|_{\S}=\n_\nu V|_\S=-(1-c\rho)|_{\S}= -1.
\end{align}
By the weighted Reilly-type formula \eqref{xeq0} applied to ${V}_\ve$ and the boundary information \eqref{info-bdry}, we get
\begin{align}\label{eq-Reilly1}
-\int_\O V_\ve |\n^2 f|^2 dv
= -\int_\S |\n_\S f|^2da +\int_\O \left(\n^2 V_\ve -\Delta V_\ve g+V_\ve \mathrm{Ric}_g\right)(\n f, \n f)dv.
\end{align}
On the other hand, by the Pohozaev identity \eqref{poho} applied to $X=\n V_\ve$ and the boundary information \eqref{info-bdry} again, we obtain
\begin{align}\label{eq-Pohozaev1}
\int_\S |\n_\S f|^2  da=\int_\S (\p_\nu f)^2 da +\int_\O (2\n^2 V_\ve-\Delta V_\ve g)(\n f, \n f) dv.
\end{align}
Combining \eqref{eq-Reilly1} and \eqref{eq-Pohozaev1}, we have
\begin{align}\label{comb}
\int_\S (\p_\nu f)^2 da&=\int_\O \left(-\n^2 V_\ve(\n f,\n f)+ V_\ve |\n^2 f|^2+ V_\ve \mathrm{Ric} (\n f, \n f)\right)dv.
\end{align}
By the curvature condition \eqref{curv-cond} and Proposition \ref{prop-smoothing}, we deduce
\begin{equation*}
\int_\S (\p_\nu f)^2 da\geq (c-\ve)\int_\O |\n f|^2dv.
\end{equation*}
By letting $\ve\to 0$, we get \eqref{key-ineq1}.

For \eqref{key-ineq2}, we need only to look at \eqref{eq-Reilly1}. Because of $\n^2 V_\ve\le -(c-\ve)g$, we deduce that
$$\n^2 V_\ve -\Delta V_\ve  g\ge (n-1)(c-\ve)g.$$
It follows from \eqref{eq-Reilly1} that
\begin{align*}
\int_\S |\n_\S f|^2da \ge (n-1)(c-\ve)\int_\O |\n f|^2 dv.
\end{align*}
By letting $\ve\to 0$, we get \eqref{key-ineq2}.
\end{proof}

\noindent{\it Proof of Theorem \ref{thm1}.}
Let $f$ be a Steklov eigenfunction corresponding to $\sigma_1$. Then we have
\begin{eqnarray}
&&\int_\S (\pt_\nu f)^2 da=\sigma_1^2\int_\S f^2 da,\label{xxeq1}\\
&&\int_\O |\n f|^2 dv= \sigma_1 \int_\S f^2da.\label{xxeq2}
\end{eqnarray}
Combining the above two identities with \eqref{key-ineq1}, we get $$\s_1\ge c.$$


Next we consider the case $\sigma_1=c$. First we have the following observation.
\begin{prop}\label{lem-rigidity}
If $\sigma_1=c$, then
\begin{align}\label{eq-rigidity}
\n^2 f&=0,\quad \mathrm{Ric}_g(\n f,\n f)=0\hbox{ on }\O.
\end{align}
\end{prop}
\begin{proof}
Recall that $V_\ve$ is constructed by the Riemannian convolution and the gluing. By Lemma~3 (2) in Greene and Wu's \cite{GW76}, we know $$\widetilde{V}_\tau\to V\hbox{ uniformly on } \overline{O_2} \hbox{ as }\tau\to 0.$$Also $V_\ve=V$ on $\O\setminus  \overline{O_2}$.
Hence $$V_\ve\to V\hbox{ uniformly on } \O \hbox{ as }\ve\to 0.$$

Because $\sigma_1=c$, we see from \eqref{xxeq1}, \eqref{xxeq2} and \eqref{comb} that
\begin{eqnarray*}
&&c\int_\O |\n f|^2 dv=\int_\S (\p_\nu f)^2 da
\\& \ge& \int_\O \left((c-\ve)|\n f|^2+ V_\ve |\n^2 f|^2+ V_\ve \mathrm{Ric} (\n f, \n f)\right)dv.
\end{eqnarray*}
Letting $\ve\to 0$, we get
$$ \int_\O (V |\n^2 f|^2+V \mathrm{Ric} (\n f, \n f)) dv =0.$$
We get the conclusion.


\end{proof}

By Proposition~\ref{lem-rigidity}, the nontrivial function $f$ satisfies
\begin{equation*}
\n^2 f=0 \hbox{ in }\O, \quad \p_\nu f=cf\hbox{ on } \S.
\end{equation*}
Then we may apply Theorem~19 in \cite{RS12} to complete the proof of the equality part of Theorem~\ref{thm1}. Alternatively, here we provide a different argument Proposition~\ref{lem-rigidity1}, which is of independent interest. More importantly, our Proposition~\ref{lem-rigidity1} requires weaker assumptions. The idea of the proof of Proposition~\ref{lem-rigidity1} is due to Ben~Andrews. We are deeply grateful to him for suggesting the proof here.
\begin{prop}\label{lem-rigidity1}
Let $(\O, g)$ be an $n$-dimensional compact Riemannian manifold with boundary $\S$ such that  $\mathrm{Ric}_g\ge 0$ in $\O$ and $H\ge (n-1)c$ along $\S$. Assume there exists a nontrivial smooth function $f$ satisfying
\begin{equation*}
\n^2 f=0 \hbox{ in }\O, \quad \p_\nu f=cf\hbox{ on } \S.
\end{equation*}
Then $\O$ is isometric to  a Euclidean ball with radius $1/c$.
\end{prop}
\begin{proof}
Since the result for the case $n=2$ has been proved by Escobar \cite{Esc97} (see pages 548--549 there), we consider the case $n\ge 3$ in the following. Note $\nabla^2 f=0$. So $|\nabla f|$ is constant on $\O$. Without loss of generality, assume $|\nabla f|=1$. Then we can check that any level set
\begin{equation*}
M_t:=\{x\in \O:f(x)=t\}
\end{equation*}
in $\O$ is totally geodesic with unit normal $\nabla f$. Moreover, $\nabla f$ is a global Killing vector field. Therefore, $\O$ is of a warped product structure. In other words, $\Omega$ is contained in a Riemannian direct product
\begin{equation}
\widehat{M}:=M_0\times \R,
\end{equation}
with $f=t$ being the coordinate for $\R$. In particular, in the coordinate of $M_0\times \R$, $\n f=(0, 1)$. Compare Brinkmann's result in 1925; see Theorem~4.3.3 in \cite{Pet16}.

Next, since $\int_\S f da=0$, we know $f$ changes sign on $\S$ and in turn $M_0$ decomposes the boundary $\S$ of $\O$ into two parts, the upper part $\S_+=\{x\in \S: f(x)\geq 0\}$ and the lower part $\S_-=\{x\in \S: f(x)\leq 0\}$.

Let $\tilde{M_0}$ be a connected component of $M_0$.
By virtue of  the fact $\n f=(0,1)$ and the boundary condition $\pt_\nu f=cf$, the connected components $\tilde{\S}_\pm$ can be written as two graphs over $\tilde{M_0}$, i.e.,
\begin{align*}
\tilde{\S}_\pm&:=\{(x,u_\pm(x)):x\in \tilde{M_0}\},
\end{align*}
where $u_\pm: \tilde{M_0}\rightarrow \R$ are the corresponding graph functions. Note that $u_+\geq 0$ and $u_-\leq 0$.

First let us focus on $\tilde{\S}_+$. A standard computation shows that the outer unit normal $\nu$ of $\tilde{\S}_+$ reads
\begin{equation*}
\nu=\frac{(-\nabla_{\tilde{M_0}} u_+,1)}{\sqrt{|\nabla_{\tilde{M_0}} u_+|^2+1}}.
\end{equation*}
Then $\pt _\nu f=cf$ and $\nabla f=(0,1)$ gives us
\begin{equation}\label{eq-graph}
\frac{1}{\sqrt{|\nabla_{\tilde{M_0}} u_+|^2+1}}=cu_+.
\end{equation}

Moreover, from \eqref{eq-graph} we observe some properties on the graph function $u_+$. First, $u_+$ is continuous on $\tilde{M_0}$ and has the range $[0,1/c]$ with $u_+|_{\pt \tilde{M_0}}=0$. Second, $u_+$ is smooth away from $\pt \tilde{M_0}$, since the boundary $\S$ is smooth. Third, the set $A:=\{x\in \tilde{M_0}:u_+(x)=1/c\}$ is a compact set in $\tilde{M_0}$. 

Using $u_+$, we define
\begin{equation}
v(x):=\frac{1}{c}\sqrt{1-c^2 u_+^2(x)},\quad x\in \tilde{M_0}.
\end{equation}
In view of the properties above on $u_+$, we have the corresponding ones for $v$. First, $v$ is continuous on $\tilde{M_0}$ and $v\in [0,1/c]$ with $v|_{\pt \tilde{M_0}}=1/c$. Second, $v$ is smooth at any $x$ with $v(x)\in (0,1/c)$. Third, $\{x\in \tilde{M_0}: v(x)=0\}=A$. (So $v$ is smooth on $\tilde{M_0}\setminus (\pt \tilde{M_0}\cup A)$.) Furthermore, by \eqref{eq-graph}, we have an important additional property $|\nabla_{\tilde{M_0}} v|=1$ on $\tilde{M_0}\setminus (\pt \tilde{M_0}\cup A)$.

Next let us study the level set of $v$. Define
\begin{equation*}
T_\tau:=\{x\in \tilde{M_0}:v(x)=\tau\},\quad \tau \in [0,1/c].
\end{equation*}
Note that $T_0=A$ and $T_{1/c}=\pt \tilde{M_0}$. First we claim
\begin{equation}\label{eq-claim}
\tau\leq {\rm dist}(T_0,T_\tau),\quad \tau \in [0,1/c].
\end{equation}
To prove the claim, fix any $\tau\in (0,1/c]$. Let $\gamma:[0,{\rm dist}(T_0,T_\tau)]\rightarrow \tilde{M_0}$ be the arc-length parametrized minimizing geodesic achieving the distance ${\rm dist}(T_0,T_\tau)$ with $\gamma(0)\in T_0$ and $\gamma({\rm dist}(T_0,T_\tau))\in T_\tau$. We can check that $\gamma$ minus the two end points has no intersections with $T_0\cup T_{1/c}$. So $v$ is smooth on $\gamma$ minus the two end points. Then we have
\begin{align*}
\tau&=\tau-0=\lim_{\ve\rightarrow 0+} v(\gamma(s))\big|_\ve^{{\rm dist}(T_0,T_\tau)-\ve}\\
&=\lim_{\ve\rightarrow 0+}\int_\ve^{{\rm dist}(T_0,T_\tau)-\ve} \frac{d}{ds}(v(\gamma(s)))ds\\
&= \lim_{\ve\rightarrow 0+} \int_\ve^{{\rm dist}(T_0,T_\tau)-\ve}\langle \nabla_{\tilde{M_0}} v,\gamma'\rangle ds\\
&\leq {\rm dist}(T_0,T_\tau),
\end{align*}
where we used $\langle \nabla_{\tilde{M_0}} v,\gamma'\rangle\leq |\nabla_{\tilde{M_0}} v|=1$. So we have proved the claim \eqref{eq-claim}. In particular, we have
\begin{equation}\label{eq-dist}
1/c\leq {\rm dist}(T_0,T_{1/c})={\rm dist}(A,\pt \tilde{M_0}).
\end{equation}

Now we intend to use the result in \cite{Li14} to conclude that $\tilde{M_0}$ is a Euclidean ball with radius $1/c$. We proceed as follows.

First, we set $e_1=\pt_t$ and take an orthonormal basis $\{e_i\}_{i=2}^{n}$ for $T\tilde{M_0}$. Then $\{e_i\}_{i=1}^{n}$ is an orthonormal basis for $T\widehat{M}$. Since $\tilde{M_0}$ is totally geodesic, by the Gauss equation, we know that the Riemannian curvature of $\tilde{M_0}$ satisfies
\begin{equation*}
R^{\tilde{M_0}}_{ijij}=R^{\widehat{M}}_{ijij},\quad 2\leq i,j\leq n.
\end{equation*}
On the other hand, by the Ricci identity, we obtain
\begin{equation*}
0=f_{ijk}-f_{ikj}=\sum_{p=1}^{n}f_pR^{\widehat{M}}_{pijk}, \quad 1\leq i,j,k\leq n,
\end{equation*}
which implies $R^{\widehat{M}}_{1ijk}=0$ for $1\leq i,j,k\leq n$. Therefore, we can deduce
\begin{equation}
\mathrm{Ric}^{\tilde{M_0}}(e_i,e_i)=\mathrm{Ric}^{\widehat{M}}(e_i,e_i)\geq 0, \quad 2\leq i\leq n.
\end{equation}

Second, we can prove that the second fundamental form of $\pt \tilde{M_0}$ in $\tilde{M_0}$, denoted by $h_{\p \tilde{M_0}}$, satisfies $h_{\p \tilde{M_0}}=h|_{\pt \tilde{M_0}}$. In fact, for any point $p\in \pt \tilde{M_0}$, by $\pt_\nu f(p)=cf(p)=0$, we know $\nabla f(p)=(0,1)\in T_p \S$. In a neighborhood of $p$ choose an orthonormal local frame $\{e_i\}_{i=1}^{n-1}$ of $\S$ such that $e_1(p)=\nabla f(p)=(0,1)$. Since $\tilde{M_0}$ is totally geodesic with constant unit normal $\nabla f$, we know that $\{e_i\}_{i=2}^{n-1}$ is an orthonormal local frame for $\pt \tilde{M}_0$ and $\nu$ is the unit outward normal of $\pt \tilde{M_0}$ in $\tilde{M_0}$. Moreover, for any $1\leq i\leq n-1$, we have
\begin{align*}
0&=\nabla^2 f(e_i,\nu)=e_i(\pt_\nu f)-\langle \nabla_{e_i}\nu,\nabla f\rangle=c e_i(f)-h_{ij}e_j(f),
\end{align*}
from which we get that $h_{11}=c$ and $h_{1i}=0$ for $2\leq i\leq n-1$. So $e_1$ is a principal direction of $\S$ at $p$ corresponding to the principal curvature $c$. Now noting again that $\tilde{M_0}$ is totally geodesic, we have
\begin{align*}
h_{\p \tilde{M_0}}(e_i,e_j)&=\langle \nabla^{\tilde{M_0}}_{e_i}e_j,\nu\rangle=\langle \nabla^{\widehat{M}}_{e_i}e_j,\nu\rangle=h(e_i,e_j),\quad 2\leq i,j\leq n-1,
\end{align*}
which is the claimed $h_{\p \tilde{M_0}}=h|_{\pt \tilde{M_0}}$. Thus we get
\begin{equation*}
H_{\pt \tilde{M_0}}=\mathrm{tr}_{g_{\pt \tilde{M_0}}} h_{\pt \tilde{M_0}}=\mathrm{tr}|_{g_{\pt \tilde{M_0}}} h=H-c\geq (n-2)c.
\end{equation*}

Now, since $\mathrm{Ric}^{\tilde{M_0}}\ge 0$ and $H_{\pt \tilde{M_0}}\ge (n-2)c$, we can use Theorem~1.1 in \cite{Li14} to conclude that
\begin{eqnarray}\label{M-Li}
\sup_{x\in \tilde{M_0}}{\rm dist}(x, \p \tilde{M_0})\le \frac{1}{c},
\end{eqnarray}
 with the equality holding if and only if $\tilde{M_0}$ is isometric to an $(n-1)$-dimensional Euclidean ball with radius $1/c$.  Combining it with \eqref{eq-dist}, we see the equality in \eqref{M-Li} holds true. Hence $\tilde{M_0}$ is an $(n-1)$-dimensional Euclidean ball with radius $1/c$ and centered at $A$, a single point set $\{x_0\}$, up to an isometry. Without loss of generality, we assume $\tilde{M_0}$ is exactly a Euclidean ball.

Next for $\tau\in (0,1/c)$, we prove that $T_\tau$ coincides with the geodesic sphere in $\tilde{M_0}$ with radius $\tau$ and centered at $x_0$, which is denoted by $S_\tau$. On the one hand, by \eqref{eq-claim}, we know that $T_\tau$ lies outside of $S_\tau$. On the other hand, by the similar argument as in the proof of \eqref{eq-claim}, we can prove $c^{-1}-\tau\leq {\rm dist}(T_\tau,T_{1/c})$. So $T_\tau$ lies inside of $S_\tau$. As a consequence, $T_\tau=S_\tau$ as desired. It follows that $v(x)=|x-x_0|$ and $u_+^2(x)+|x-x_0|^2=1/c^2$. Thus $\tilde{\S}_+$ is a hemisphere.

We can apply the same argument to $\tilde{\S}_-$ to conclude that $\tilde{\S}_-$ is also a hemisphere. Hence $\S$ is a union of disjoint Euclidean spheres.
Note from Remark \ref{rem-connect} that $\S$ is connected. Hence $\S$ must be one Euclidean sphere. We finish the proof of Proposition~\ref{lem-rigidity1} as well as the proof of Theorem~\ref{thm1}.
\end{proof}

\begin{rem}\label{key-rem}
We would like to call the reader's attention to a recent paper by Chen--Lai--Wang \cite{CLW19} where Obata-type theorems for general Robin boundary conditions were investigated.
\end{rem}

\noindent{\it Proof of Theorem~\ref{thm2}.}
First we prove \eqref{upper}. Let $z$ be an eigenfunction corresponding to the first nonzero eigenvalue $\lambda_1$ of $\Delta_\S$ and $f$ be its harmonic extension to the interior of $\O$.
Then we have
\begin{eqnarray*}
\int_\S |\n_\S z|^2 da= \l_1 \int_\S z^2 da.
\end{eqnarray*}
Using \eqref{key-ineq2}, we have
\begin{eqnarray}\label{2-pf}
 \l_1 \int_\S z^2 da=\int_\S |\n_\S z|^2 da\ge (n-1)c\int_\O |\n  f|^2 dv.
\end{eqnarray}
On the other hand, by using the variational characterization \eqref{var-ch} for $\sigma_1$, we have
\begin{align}\label{2-pf1}
\s_1\le \frac{ \int_\O |\n f|^2dv}{\int_\S z^2da}.
\end{align}
Combining \eqref{2-pf} and \eqref{2-pf1}, we get the assertion $$\s_1\le \frac{\l_1}{(n-1)c}.$$
Now we consider the equality $\sigma_1=  \lambda_1/((n-1)c)$. By the above deduction, we know that
$f$ is indeed a Steklov eigenfunction corresponding to $\s_1$.

By the similar argument as in Proposition~\ref{lem-rigidity}, we have
\begin{align*}
\n^2 f&=0,\quad \mathrm{Ric}_g(\n f,\n f)=0.
\end{align*}
Therefore, taking $\{e_i\}_{i=1}^{n-1}$ as a local orthonormal frame on $\S$, we have
\begin{equation*}
0=\sum_{i=1}^{n-1}\nabla^2 f(e_i,e_i)=\Delta_\Sigma f+H \pt_\nu f=-\lambda_1 f+H \sigma_1  f,
\end{equation*}
which together with $\lambda_1=  (n-1)c\sigma_1$ implies $H=(n-1)c$. Note $H\geq (n-1)c$. So we have $h=cg_\S$.

Next we compute for any $1\leq i\leq n-1$,
\begin{align*}
0&=\nabla^2 f(e_i,\nu)=e_i(\pt_\nu f)-\langle \nabla_{e_i}\nu,\nabla f\rangle=\sigma_1 e_i(f)-h_{ij}e_j(f).
\end{align*}
So we get $\sigma_1=c$. Then the proof reduces to that of Proposition~\ref{lem-rigidity1} and thus is complete.

\medskip

Lastly we prove \eqref{high}. Recall the min-max variational characterizations of the two eigenvalue problems, i.e.,
\begin{equation}
\sigma_j=\inf_{\substack{U\subset H^1(\Omega),\\dim\, U=j+1}}\sup_{\substack{ 0\neq u\in U,\\ \int_\Sigma u^2da=1}}\int_\Omega |\nabla u|^2 dv,
\end{equation}
for all $j\geq 0$, and
\begin{equation}\label{VCL}
\lambda_j=\inf_{\substack{U\subset H^1(\Sigma),\\dim\, U=j+1}}\sup_{ \substack{0\neq u\in U,\\ \int_\Sigma u^2da=1}}\int_\Sigma |\nabla_\Sigma u|^2 da,
\end{equation}
for all $j\geq 0$.

Let $\{\varphi_k\}_{k=0}^\infty$ be a complete orthonormal basis of $L^2(\S)$ such that $\varphi_k$ is an eigenfunction corresponding to $\l_k$ for $\Delta_\S$. For each $\varphi_k$, let $f_k$ be its harmonic extension to $\O$. Therefore by using \eqref{key-ineq2}, we obtain
\begin{align*}
\sigma_j&\leq \sup_{\sum_{k=0}^j a_k^2=1}\int_\O \left|\n \left(\sum_{k=0}^j a_kf_k\right)\right|^2dv\\
&\leq \frac{1}{(n-1)c} \sup_{\sum_{k=0}^j a_k^2=1}\int_\S \left|\n_\S \left(\sum_{k=0}^j a_k\varphi_k\right)\right|^2da\\
&=\frac{1}{(n-1)c} \sup_{\sum_{k=0}^j a_k^2=1}\sum_{k=0}^j a_k^2 \lambda_k\\
&\le\frac{\lambda_j}{(n-1)c}.
\end{align*}
The proof is completed.

\

\noindent{\bf Acknowledgements.} The authors would like to thank Ben~Andrews for valuable discussions. They would also like to thank Dr.~Linlin~Sun for pointing out an inaccurate statement on the comparison of our Theorem~\ref{thm2} and the known results in a previous version.


\bibliographystyle{Plain}

\begin{thebibliography}{10}





\bibitem{AC11} Ben~Andrews and Julie~Clutterbuck, \emph{Proof of the fundamental gap conjecture}, J. Amer. Math. Soc. \textbf{24} (2011), no.~3, 899--916.

\bibitem{CLW19}  X.~Chen, M.~Lai and F.~Wang, \emph{The Obata equation with Robin boundary condition},  Rev. Mat. Iberoam. \textbf{37} (2021), no.~2, 643--670.

\bibitem{CGH20} Bruno~Colbois, Alexandre~Girouard, and Asma~Hassannezhad, \emph{The Steklov and Laplacian spectra of Riemannian manifolds with boundary}, J. Funct. Anal. \textbf{278} (2020), no.~6, 108409, 38 pp.


\bibitem{Esc97} Jos\'{e}~F.~Escobar, \emph{The geometry of the first non-zero Stekloff eigenvalue}, J. Funct. Anal. \textbf{150} (1997), no.~2, 544--556.

\bibitem{Esc99} Jos\'{e}~F.~Escobar, \emph{An isoperimetric inequality and the first Steklov eigenvalue}, J. Funct. Anal. \textbf{165} (1999), no.~1, 101--116.

\bibitem{Gho02} Mohammad~Ghomi, \emph{The problem of optimal smoothing for convex functions}, Proc. Amer. Math. Soc. \textbf{130} (2002), no.~8, 2255--2259.

\bibitem{GKLP22} Alexandre~Girouard, Mikhail~Karpukhin, Michael~Levitin, and Iosif~Polterovich, \emph{The Dirichlet-to-Neumann map, the boundary Laplacian, and H\"{o}rmander's rediscovered manuscript}, J. Spectr. Theory \textbf{12} (2022), no.~1, 195--225.

\bibitem{GP17} A.~Girouard and I.~Polterovich, \emph{Spectral geometry of the Steklov problem}, J. Spectr. Theory \textbf{7} (2017), no.~2, 321--359.

\bibitem{GW72-73} R.~E.~Greene and H.~Wu, \emph{On the subharmonicity and plurisubharmonicity of geodesically convex functions}, Indiana Univ. Math. J. \textbf{22} (1972/73), 641--653.

\bibitem{GW76} R.~E.~Greene and H.~Wu, \emph{$C^\infty$ convex functions and manifolds of positive curvature}, Acta Math. \textbf{137} (1976), no.~3-4, 209--245.

\bibitem{GW79} R.~E.~Greene and H.~Wu, \emph{$C^\infty$ approximations of convex, subharmonic, and plurisubharmonic functions}, Ann. Sci. \'{E}cole Norm. Sup. (4) \textbf{12} (1979), no.~1, 47--84.

\bibitem{HK78} E.~Heintze and H.~Karcher, \emph{A general comparison theorem with applications to volume
estimates for submanifolds,} Ann. Sci. \'Ecole Norm. Sup. \textbf{11}(1978), 451--470.

\bibitem{Ich81} Ryosuke~Ichida, \emph{Riemannian manifolds with compact boundary}, Yokohama Math. J. \textbf{29} (1981), no.~2, 169--177.

\bibitem{IT01} Jin-ichi~Itoh and Minoru~Tanaka, \emph{The Lipschitz continuity of the distance function to the cut locus}, Trans. Amer. Math. Soc. \textbf{353} (2001), no.~1, 21--40.

\bibitem{Kar17} Mikhail~A.~Karpukhin, \emph{Bounds between Laplace and Steklov eigenvalues on nonnegatively curved manifolds}, Electron. Res. Announc. Math. Sci. \textbf{24} (2017), 100--109.





\bibitem{Kas82} Atsushi~Kasue, \emph{A Laplacian comparison theorem and function theoretic properties of a complete Riemannian manifold}, Japan. J. Math. (N.S.) \textbf{8} (1982), no.~2, 309--341.

\bibitem{Kas83} Atsushi~Kasue, \emph{Ricci curvature, geodesics and some geometric properties of Riemannian manifolds with boundary}, J. Math. Soc. Japan \textbf{35} (1983), no.~1, 117--131.

\bibitem{Kro92} Pawel~Kr\"{o}ger, \emph{On the spectral gap for compact manifolds}, J. Differential Geom. \textbf{36} (1992), no.~2, 315--330.

\bibitem{KKK14} N.~Kuznetsov, T.~Kulczycki, M.~Kwa\'{s}nicki, A.~Nazarov, S.~Poborchi, I.~Polterovich and B.~Siudeja, \emph{The legacy of Vladimir Andreevich Steklov}, Notices Amer. Math. Soc. \textbf{61} (2014), no.~1, 9--22.

\bibitem{Li14} Martin~Man-chun~Li, \emph{A sharp comparison theorem for compact manifolds with mean convex boundary}, J. Geom. Anal. \textbf{24} (2014), no.~3, 1490--1496.

\bibitem{MM88} Mario~J.~Micallef and John~Douglas~Moore, \emph{Minimal two-spheres and the topology of manifolds with positive curvature on totally isotropic two-planes}, Ann. of Math. (2) \textbf{127} (1988), no. 1, 199--227.


\bibitem{MW93} Mario~J.~Micallef and McKenzie~Y.~Wang, \emph{Metrics with nonnegative isotropic curvature}, Duke Math. J. \textbf{72} (1993), no. 3,  649--672.



\bibitem{Mon13} \'{O}.~Monta\~{n}o Carre\~{n}o, \emph{The Stekloff problem for rotationally invariant metrics on the ball}, Rev. Colombiana Mat. \textbf{47} (2013), no.~2, 181--190.

\bibitem{Mon16} \'{O}.~Monta\~{n}o Carre\~{n}o, \emph{Escobar's Conjecture for the First Steklov Eigenvalue on $n$-ellipsoids}, Revista de Ciencias \textbf{20} (2016), no.~2, 55--61.

\bibitem{Nor94} Maria~Helena~Noronha, \emph{Self-duality and $4$-manifolds with nonnegative curvature on totally isotropic $2$-planes}, Michigan Math. J. \textbf{41} (1994), no.~1, 3--12.


\bibitem{Pay70} L.~E.~Payne, \emph{Some isoperimetric inequalities for harmonic functions}, SIAM J. Math. Anal. \textbf{1} (1970), 354--359.

\bibitem{PW60} L.~E.~Payne and H.~F.~Weinberger, \emph{An optimal Poincar\'{e} inequality for convex domains}, Arch. Rational Mech. Anal. \textbf{5} (1960), 286--292.


\bibitem{Pet16} Peter~Petersen, \emph{Riemannian geometry, Third edition}, Graduate Texts in Mathematics, \textbf{171}, Springer, Cham, 2016.


\bibitem{PS19} Luigi~Provenzano and Joachim~Stubbe, \emph{Weyl-type bounds for Steklov eigenvalues}, J. Spectr. Theory \textbf{9} (2019), no.~1, 349--377.

\bibitem{QX15} Guohuan~Qiu and Chao~Xia, \emph{A generalization of Reilly's formula and its applications to a new Heintze-Karcher type inequality}, Int. Math. Res. Not. IMRN 2015, no.~17, 7608--7619.


\bibitem{RS12} S.~Raulot and A.~Savo, \emph{On the first eigenvalue of the Dirichlet-to-Neumann operator on forms}, J. Funct. Anal. \textbf{262} (2012), no.~3, 889--914.

\bibitem{Rei77} Robert~C.~Reilly, \emph{Applications of the Hessian operator in a Riemannian manifold}, Indiana Univ. Math. J. \textbf{26} (1977), no.~3, 459--472.

\bibitem{Ste02} W.~Stekloff, \emph{Sur les probl\`{e}mes fondamentaux de la physique math\'{e}ématique}, Ann. Sci. \'{E}cole Norm. Sup. (3) \textbf{19} (1902), 191--259.


\bibitem{Tay96} M.~E.~Taylor, \emph{Partial differential equations. II}, Applied Mathematical Sciences, 116, Springer-Verlag, New York, 1996.

\bibitem{WX09} Qiaoling~Wang and Changyu~Xia, \emph{Sharp bounds for the first non-zero Stekloff eigenvalues}, J. Funct. Anal. \textbf{257} (2009), no.~8, 2635--2644.


\bibitem{Wu79} H.~Wu, \emph{An elementary method in the study of nonnegative curvature}, Acta Math. \textbf{142} (1979), no.~1-2, 57--78.



\bibitem{Xio18} Changwei~Xiong, \emph{Comparison of Steklov eigenvalues on a domain and Laplacian eigenvalues on its boundary in Riemannian manifolds}, J. Funct. Anal. \textbf{275} (2018), no.~12, 3245--3258.

\bibitem{Xio19} Changwei~Xiong, \emph{On the spectra of three Steklov eigenvalue problems on warped product manifolds}, J. Geom. Anal. \textbf{32} (2022), no.~5, Paper No. 153, 35 pp.

\bibitem{YY17} Liangwei~Yang and Chengjie~Yu, \emph{A higher dimensional generalization of Hersch--Payne--Schiffer inequality for Steklov eigenvalues}, J. Funct. Anal. \textbf{272} (2017), no.~10, 4122--4130.

\bibitem{ZY84} Jia~Qing~Zhong and Hong~Cang~Yang, \emph{On the estimate of the first eigenvalue of a compact Riemannian manifold}, Sci. Sinica Ser. A \textbf{27} (1984), no.~12, 1265--1273.

\end{thebibliography}

\end{document}